\documentclass[12pt,twoside]{amsart}
\usepackage{geometry}
\usepackage{amssymb,amsmath,amsthm, amscd, enumerate, mathrsfs}
\usepackage{graphicx, hhline}
\usepackage[all]{xy}
\usepackage[dvipdfmx]{hyperref}

\title{ACC for log canonical thresholds for complex analytic spaces}
\author{Osamu Fujino}
\dedicatory{Dedicated to Professor Vyacheslav V.~Shokurov on the 
occasion of his seventieth birthday}
\date{2022/8/25, version 0.01}
\subjclass[2010]{Primary 14E30; Secondary 32S05}
\keywords{log canonical thresholds, ascending chain condition, 
log canonical singularities, minimal model program, complex analytic 
singularities}
\address{Department of 
Mathematics, Graduate School of Science, 
Kyoto University, Kyoto 606-8502, Japan}
\email{fujino@math.kyoto-u.ac.jp}

\DeclareMathOperator{\lct}{lct}
\DeclareMathOperator{\Supp}{Supp}
\DeclareMathOperator{\LCT}{LCT}
\newtheorem{thm}{Theorem}[section]
\newtheorem{lem}[thm]{Lemma}

\theoremstyle{definition}
\newtheorem{defn}[thm]{Definition}
\newtheorem{rem}[thm]{Remark}
\newtheorem{ex}[thm]{Example}
\newtheorem*{ack}{Acknowledgments}  

\makeatletter
    
    \@addtoreset{equation}{section}
\makeatother
\begin{document}

\maketitle 

\begin{abstract} 
We show that log canonical thresholds for complex 
analytic spaces satisfy the ACC.  
\end{abstract}

\section{Introduction}

As usual, {\em{ACC}} stands for the {\em{ascending chain 
condition}} and {\em{DCC}} stands for the {\em{descending 
chain condition}}. 
In \cite{hmx}, the ACC for log canonical thresholds, 
which was conjectured by Shokurov, was completely 
settled for algebraic varieties. 
We note that Shokurov raised many 
conjectures that assert the ascending or 
descending chain condition for various 
naturally defined invariants coming from algebraic geometry (see, for 
example, \cite{shokurov-problem}, \cite{shokurov1}, 
\cite[Chapter 18]{kollar-et}, \cite[Section 8]{kollar1}, 
and so on). 
In this paper, we generalize it for complex analytic spaces. 

Let us start with the definition of {\em{log canonical thresholds}} 
for complex analytic spaces. Note that 
$X$ is a normal complex analytic space in Definition \ref{def1.1}. 
For various aspects of log canonical 
thresholds, we strongly recommend the 
reader to see \cite[Sections 8, 9, and 10]{kollar1}. 

\begin{defn}[Log canonical thresholds for 
complex analytic spaces]\label{def1.1}
Let $(X, \Delta)$ be a log canonical pair and 
let $M$ be an effective $\mathbb R$-Cartier $\mathbb R$-divisor 
on $X$. 
Let $c$ be a nonnegative real number such that 
$(X, \Delta+cM)$ is log canonical 
and that there exists a non-kawamata 
log terminal center of $(X, \Delta+cM)$ which is contained 
in $\Supp M$. 
Then $c$ is called the {\em{log canonical threshold}} of 
$M$ with respect to $(X, \Delta)$ and is 
usually denoted by $\lct (X, \Delta; M)$. 
When $M=0$, we put $\lct(X, \Delta; M)=+\infty$. 
\end{defn}

The following definition and example show 
the reason why we adopt the above definition 
of log canonical thresholds 
for complex analytic spaces, which looks slightly different from 
the usual definition of log canonical thresholds for 
algebraic varieties. 

\begin{defn}\label{def1.2} 
Let $X$ be a normal complex variety. A {\em{prime divisor}} 
on $X$ is an irreducible and reduced closed subvariety 
of codimension one. 
An {\em{$\mathbb R$-divisor}} $D$ on $X$ 
is a {\em{locally finite}} formal sum 
\begin{equation*}
D=\sum _i a_i D_i, 
\end{equation*} 
where $D_i$ is a prime divisor on $X$ with 
$a_i\in 
\mathbb R$ for every $i$ and 
$D_i\ne D_j$ for $i\ne j$. 
When $a_i\in \mathbb Q$ holds for every $i$, 
$D$ is called a {\em{$\mathbb Q$-divisor}} 
on $X$. 

Let $D$ be an $\mathbb R$-divisor 
on a normal complex variety $X$ and let 
$x$ be a point of $X$. 
If $D$ is written as a finite $\mathbb R$-linear 
(resp.~$\mathbb Q$-linear) combination 
of Cartier divisors on some open neighborhood of $x$, 
then $D$ is said to be {\em{$\mathbb R$-Cartier 
at $x$}} (resp.~{\em{$\mathbb Q$-Cartier at $x$}}). 
If $D$ is $\mathbb R$-Cartier 
(resp.~$\mathbb Q$-Cartier) at $x$ for every $x\in X$, 
then $D$ is said to be {\em{$\mathbb R$-Cartier}} 
(resp.~{\em{$\mathbb Q$-Cartier}}). 
\end{defn}

\begin{ex}\label{ex1.3} 
We consider $X=\mathbb C$. 
Let $\{P_n\}_{n\in \mathbb Z_{>0}}$ be a 
set of mutually distinct discrete points of $X$. 
We put $M=\sum _{n\in \mathbb Z_{>0}} \frac{n-1}{n}P_n$. 
Then $M$ is a $\mathbb Q$-Cartier $\mathbb Q$-divisor 
on $X$. 
In this case, $(X, M)$ is log canonical and 
$(X, tM)$ is not log canonical for every positive 
real number $t$ with $t>1$. 
However, there are no non-kawamata log terminal centers 
of $(X, M)$, that is, $(X, M)$ is kawamata log terminal. 
\end{ex}

We note an obvious remark. 

\begin{rem}\label{rem1.4}(1) 
If $(X, \Delta)$ and $M$ are both algebraic in 
Definition \ref{def1.1}, 
then it is easy to see that 
the following equality 
\begin{equation*}
\lct(X, \Delta; M)=\sup \{t\in \mathbb R\, |\, 
{\text{$(X, \Delta+tM)$ is log canonical}}\}
\end{equation*} 
holds. 

(2) In Definition \ref{def1.1}, let $U$ be a relatively 
compact open subset of $X$. 
Then we can check that 
\begin{equation*}
\lct(U, \Delta|_U; M|_U)=\sup \{t\in \mathbb R\, |\, 
{\text{$(U, \Delta|_U+tM|_U)$ is log canonical}}\} 
\end{equation*} 
holds by using the resolution of singularities. 
\end{rem}

By Remark \ref{rem1.4} (1), $\lct(X, \Delta; M)$ coincides with 
the usual one when $(X, \Delta)$ and $M$ are all algebraic. 

\begin{defn}\label{def1.5}
Let $\mathfrak{T}^{an}=\mathfrak{T}_n^{an}(I)$ denote 
the set of log canonical pairs $(X, \Delta)$, where 
$X$ is a normal complex variety of dimension $n$ and the coefficients 
of $\Delta$ belong to a set $I\subset [0, 1]$. 
We put 
\begin{equation*}
\LCT^{an}_n(I, J)=\{ \lct(X, \Delta; M) \,|\, 
(X, \Delta)\in \mathfrak{T}^{an}_n(I)\}, 
\end{equation*}
where the coefficients of $M$ belong to a subset $J$ of the positive real 
numbers. 
\end{defn}

The main result of this short paper is the ACC for log canonical 
thresholds for complex analytic spaces, which is a generalization 
of \cite[Theorem 1.1]{hmx}. 

\begin{thm}[ACC for the log canonical threshold for complex 
analytic spaces]\label{thm-main}
We fix a positive integer $n$, $I\subset [0, 1]$, and a subset 
$J$ of the positive real numbers. 
If $I$ and $J$ satisfy the DCC, then $\LCT^{an}_n(I, J)$ 
satisfies the ACC. 
\end{thm}

The main ingredient of the proof of Theorem 
\ref{thm-main} is the ACC for numerically trivial pairs, 
which is nothing but \cite[Theorem 1.5]{hmx} 
(see Theorem \ref{thm1.7}), 
and the minimal model 
program for projective morphisms between complex analytic 
spaces established in \cite{fujino-analytic}. 
Note that one of the motivations of \cite{fujino-analytic} 
is to make the minimal model program applicable to the study 
of germs of complex analytic singularities. 
We also note that a similar result was obtained 
independently by Das, Hacon, and P\u aun 
(see \cite[Theorem 6.4]{dhp}). 

\begin{thm}[{ACC for numerically trivial pairs, 
see \cite[Theorem 1.5]{hmx}}]\label{thm1.7}
Fix a positive integer $n$ and a set $I\subset [0, 1]$, which 
satisfies the DCC. Then there is a finite subset $I_0\subset I$ with 
the following property: 

If $(X, \Delta)$ is an $n$-dimensional projective log 
canonical pair such that $K_X+\Delta$ is 
numerically trivial and that the coefficients of $\Delta$ 
belong to $I$, then the coefficients of $\Delta$ belong to 
$I_0$. 
\end{thm}

We note that de Fernex, Ein, and Musta\c{t}\u{a} established a striking 
result on Shokurov's ACC conjecture before \cite{hmx}. 
Here we only explain a very special case. 
For the details and some related topics, 
see \cite{dfem}, \cite{kollar2}, \cite{totaro}, and so on. 

\begin{defn}[Log canonical thresholds of 
holomorphic functions]\label{def1.8}
Let $f$ be a holomorphic function in a neighborhood 
of $0\in \mathbb C^n$. 
The {\em{log canonical threshold}} 
of $f$ at $0$ is the number $c=\lct_0(f)$ such that 
\begin{itemize}
\item $|f|^{-s}$ is $L^2$ in a neighborhood 
of $0$ for $s<c$, and 
\item $|f|^{-s}$ is not $L^2$ in a neighborhood 
of $0$ for $s>c$. 
\end{itemize}
Hence, if $f(0)\ne 0$, then $\lct_0(f)=+\infty$. 

We put 
\begin{equation*}
\mathcal H\mathcal T_n:=\{\lct_0(f)\, |\, f\in \mathcal O_{\mathbb C^n, 0}, 
f(0)=0\} \subset \mathbb R. 
\end{equation*}
This means that $\mathcal H\mathcal T_n$ is the set of log canonical 
thresholds of all possible holomorphic functions of $n$ variables 
vanishing at $0\in \mathbb C^n$. 
\end{defn}

Then we have: 

\begin{thm}[{\cite{dfem}}]\label{thm1.9}
$\mathcal H\mathcal T_n$ satisfies the ACC. 
\end{thm}

Note that the following natural inclusion 
\begin{equation*}
\mathcal H\mathcal T_n\subset \LCT^{an}_n(\{0\}, \mathbb Z_{>0}) 
\end{equation*} 
holds. Therefore, Theorem \ref{thm1.9} is a very special 
case of Theorem \ref{thm-main}. 

\begin{ack}\label{a-ack}
The author was partially 
supported by JSPS KAKENHI Grant Numbers 
JP19H01787, JP20H00111, JP21H00974, JP21H04994. 
He would like to thank Kenta Hashizume very much 
for useful discussions and comments. 
He also thanks Masayuki Kawakita and Shunsuke Takagi. 
\end{ack}

In this paper, we will freely use \cite{fujino-analytic}. 
We always assume that complex 
analytic spaces are {\em{Hausdorff}} and {\em{second-countable}}. 
We use the standard notation of the theory of 
minimal models as in \cite{kollar-mori}, \cite{fujino-foundations}, 
and \cite{fujino-analytic}. 

\section{Proof} 
Let us start with the definition of ACC sets and DCC sets. 

\begin{defn}[{ACC sets and DCC 
sets, see \cite[3.4.~DCC sets]{hmx}}]\label{def2.1}
Let $I$ be a set of real numbers. We say that $I$ satisfies the 
{\em{ascending chain condition}} or {\em{ACC}} 
(resp.~{\em{descending chain condition}} or {\em{DCC}}) if 
it does not contain any infinite strictly increasing (resp.~decreasing) 
sequences. 

We take $I\subset [0, 1]$. 
We put 
\begin{equation*}
I_+:=\left\{ 0\right\} \bigcup 
\left\{ j\in [0, 1] \left|\, \text{$j=\sum _{p=1}^l i_p$ 
for some $i_1\ldots, i_l\in I$}\right.\right\}
\end{equation*}
and 
\begin{equation*}
D(I):=\left\{ a\in [0, 1] \left|\, \text{$a=\begin{frac}{m-1+f}{m}\end{frac}$ 
for some $m\in \mathbb Z_{>0}$ and $f\in I_+$}\right.\right\}. 
\end{equation*}
It is easy to see that $I$ satisfies the DCC if and only if $D(I)$ satisfies 
the DCC. 
\end{defn}

Without any difficulties, 
we can prove a slight modification of \cite[Lemma 5.1]{hmx} 
for complex analytic spaces by using \cite{fujino-analytic}. 

\begin{lem}\label{lem2.2}
We fix a positive integer $n$ and a set $1\in I\subset [0, 1]$. 
Assume that $(X, \Delta+\Delta')$ is an 
$(n+1)$-dimensional log canonical pair such that 
$\Delta\geq 0$, $\Delta'\geq 0$ is $\mathbb R$-Cartier, 
and the coefficients of $\Delta$, $\Delta'$ and 
$\Delta+\Delta'$ belong to $I$. 
We further assume that there exists a non-kawamata log 
terminal center $V$ of $(X, \Delta+\Delta')$ such 
that $V\subset \Supp \Delta'$ with $\dim V\leq \dim X-2$. 

Then we can construct a log canonical pair $(S, \Theta)$, 
where $S$ is a projective variety of dimension at most $n$, 
the coefficients of $\Theta$ belong to $D(I)$, 
$K_S+\Theta$ is numerically trivial, and some component 
of $\Theta$ has coefficient 
\begin{equation*}
\frac{m-1+f+kc}{m}, 
\end{equation*} 
where $m$ and $k$ are positive integers, $f\in I_+$, and 
$c\in I$ is the coefficient of some component of $\Delta'$. 
\end{lem}

The proof of \cite[Lemma 5.1]{hmx} works with only some 
minor modifications since we can always construct dlt blow-ups 
by \cite{fujino-analytic} in the complex analytic setting. 

\begin{proof}[Proof of Lemma \ref{lem2.2}] 
We can replace $V$ with a maximal (with respect to inclusion) 
non-kawamata log terminal center of $(X, \Delta+\Delta')$ 
satisfying $\dim V\leq \dim X-2$ and $V\subset \Supp \Delta'$. 
We take an analytically sufficiently general point $P$ of $V$. 
Then we take an open neighborhood 
$U$ of $P$ and a Stein compact subset $W$ of $X$ such that 
$U\subset W$ and that $\Gamma (W, \mathcal O_X)$ is 
noetherian. By \cite[Theorem 1.21]{fujino-analytic}, 
after shrinking $X$ around $W$ suitably, 
we can construct a projective 
bimeromorphic morphism $\pi\colon Y\to X$ with 
$K_Y+\Delta_Y=\pi^*(K_X+\Delta+\Delta')$ such 
that 
\begin{itemize}
\item[(a)] $(Y, \Delta_Y)$ is divisorial log terminal, 
\item[(b)] $Y$ is $\mathbb Q$-factorial over $W$, 
\item[(c)] $a(E, X, \Delta+\Delta')=-1$ holds 
for every $\pi$-exceptional divisor $E$, and 
\item[(d)] there exists a $\pi$-exceptional 
divisor $F$ on $Y$ such that $\pi(F)=V$. 
\end{itemize} 
Since $\Delta'$ is $\mathbb R$-Cartier by assumption, 
$\pi^*\Delta'$ is well-defined and is $\pi$-numerically 
trivial. 
Hence we can find $B$, which is 
an irreducible component of $\Supp \pi^{-1}_*\Delta'$, and 
a $\pi$-exceptional divisor $S$ with $S\cap B\ne \emptyset$, 
$\pi(S)=V$, and $\pi(S\cap B)=V$. 
By adjunction, we obtain 
\begin{equation*}
K_S+\Theta:=(K_Y+\Delta_Y)|_S
\end{equation*} 
such that the coefficients of $\Theta$ belong to $D(I)$ and 
some component of $\Theta$ has a coefficient 
of the form 
\begin{equation*}
\frac{m-1+f+kc}{m}, 
\end{equation*} 
where $m$ and $k$ are positive integers, $f\in I_+$, and 
$c\in I$ is the coefficient of $B$ in $\pi^{-1}_*\Delta'$. 
We take an analytically sufficiently general point $v\in V\cap U$ 
and consider the fiber over $v$. Then 
we obtain $(S_v, \Theta_v)$, which is 
divisorial log terminal with $\dim S_v\leq n$, 
such that the coefficients of $\Theta_v$ belong to $D(I)$, some 
component of $\Theta_v$ has a coefficient of the form 
\begin{equation*}
\frac{m-1+f+kc}{m} 
\end{equation*} 
as desired, 
and $K_{S_v}+\Theta_v$ is numerically trivial. 
This is what we wanted. 
\end{proof}

Let us prove Theorem \ref{thm-main}. 

\begin{proof}[Proof of Theorem \ref{thm-main}]
We assume that $c_1, c_2, \ldots \in \LCT^{an}_m(I, J)$ such that 
$c_i\leq c_{i+1}$ holds for every $i$. 
It is sufficient to prove that 
$c_i=c_{i+1}$ holds for every sufficiently large $i$. 
By definition, we can take an $n$-dimensional 
log canonical pair $(X_i, \Delta_i)$ and an effective 
$\mathbb R$-Cartier $\mathbb R$-divisor 
$M_i$ on $X_i$ such that the coefficients of $\Delta_i$ 
belong to $I$, the coefficients of $M_i$ belong to 
$J$, $(X_i, \Delta_i+c_iM_i)$ is 
log canonical, and there exists a non-kawamata log 
terminal center $V_i$ of $(X_i, \Delta_i+c_iM_i)$ with $V_i\subset 
\Supp M_i$ for every $i$. 

We put 
\begin{equation*}
K=I\cup \left\{c_i \alpha\in [0, 1]\left|\, i\in \mathbb Z_{>0}, \alpha \in 
J\right.\right\}\cup \left\{\beta+c_i \gamma \in [0, 1] 
\left| \, \text{$i\in \mathbb Z_{>0}, \beta\in I$, $\gamma \in J$}\right.
\right\} \cup \{1\}. 
\end{equation*} 
Then the coefficient of $\Delta_i$, $c_iM_i$, and 
$\Delta_i+c_iM_i$ belong to $K$. 
It is easy to see that $K$ satisfies the DCC. 
We also put 
\begin{equation*}
L=\{1-\alpha \, |\, \alpha \in I\}. 
\end{equation*} 
Then $L$ obviously satisfies the ACC. 
Hence $L\cap K$ is a finite set since $K$ satisfies the DCC. 

If $\dim V_i=n-1$, then the coefficient 
of $V_i$ in $c_iM_i$ is in the finite set 
$L\cap K$. Therefore, it is sufficient to 
treat the case when $\dim V_i\leq n-2$ 
holds for 
every $i$. Hence, from now on, we assume that 
$\dim V_i\leq n-2$ holds for every $i$. 
By Lemma \ref{lem2.2}, 
for every $i$, 
we can construct a projective 
log canonical pair $(S_i, \Theta_i)$ such that 
$\dim S_i\leq n-1$, 
the coefficients of $\Theta_i$ belong to $D(K)$, 
$K_{S_i}+\Theta_i$ is numerically trivial, and some component 
of $\Theta_i$ has coefficient 
\begin{equation*}
\frac{m_i-1+f_i+k_ic_i\alpha_i}{m_i}, 
\end{equation*} 
where $m_i$ and $k_i$ are positive integers, $f_i\in K_+$, and 
$\alpha_i\in J$. 
By Theorem \ref{thm1.7}, which is nothing but \cite[Theorem 1.5]{hmx}, 
there exists a finite subset $K_0\subset D(K)$ such that 
\begin{equation*}
\frac{m_i-1+f_i+k_ic_i\alpha_i}{m_i}\in K_0. 
\end{equation*} 
Then, by \cite[Lemma 5.2]{hmx}, 
\begin{equation*}
\{c_i \alpha_i\}_{i\in \mathbb Z_{>0}}
\end{equation*} 
is a finite set. This implies that $c_i=c_{i+1}$ holds for every sufficiently 
large $i$ since $\alpha_i\in J$ for every $i$. 

This is what we wanted, that is, $\LCT^{an}_n(I, J)$ satisfies 
the ACC. 
\end{proof}

\end{document}